\documentclass[10pt]{amsart}
\usepackage{amsmath}
\usepackage{amsfonts}
\usepackage{amssymb}
\usepackage{amsthm}
\usepackage{url}
\usepackage[all]{xy}
\usepackage{dsfont}
\usepackage{graphicx}
\usepackage{caption}
\usepackage{subcaption}
\usepackage{comment} 
\usepackage{stmaryrd}
\usepackage{hyperref}
\usepackage{todonotes}
\usepackage{color}
\usepackage{enumerate,tikz-cd}
\usepackage{fixltx2e}
\usepackage{MnSymbol}
\allowdisplaybreaks

\newcommand{\git}{\mathbin{
  \mathchoice{/\mkern-6mu/}
    {/\mkern-6mu/}
    {/\mkern-5mu/}
    {/\mkern-5mu/}}}

\numberwithin{equation}{section}

\newtheorem{proposition}{Proposition}[section]
\newtheorem{lemma}[proposition]{Lemma}
\newtheorem{theorem}[proposition]{Theorem}

\theoremstyle{definition}

\newtheorem{definition}[proposition]{Definition}
\newtheorem{example}[proposition]{Example}

\DeclareMathOperator{\Bl}{Bl}

\DeclareMathOperator{\GL}{GL}

\DeclareMathOperator{\Proj}{Proj}
\DeclareMathOperator{\Spec}{Spec}

\DeclareMathOperator{\Coh}{Coh}

\newcommand{\C}{\mathbb{C}}

\newcommand{\Q}{\mathbb{Q}}

\newcommand{\pr}{\mathbb{P}}
\renewcommand{\epsilon}{\varepsilon}

\newcommand{\M}{\mathcal{M}}

\newcommand{\scA}{\mathcal{A}}

\newcommand{\scO}{\mathcal{O}}

\newcommand{\mfk}{\mathfrak{k}}

\renewcommand{\phi}{\varphi}

\newcommand\an{^{\mathrm{an}}}
\newcommand\mi{^{-1}}

\newcommand\bbc{\mathbb{C}}

\newcommand\cC{\mathcal{C}}

\newcommand\cO{\mathcal{O}}

\title[The birational geometry of GIT quotients]{The birational geometry of GIT quotients}

\author[Ruadha\'i Dervan and R\'emi Reboulet]{Ruadha\'i Dervan and R\'emi Reboulet}

\address{Ruadha\'i Dervan, School of Mathematics and Statistics, University of Glasgow, University Place, Glasgow G12 8QQ, United Kingdom}\email{ruadhai.dervan@glasgow.ac.uk}
\address{R\'emi Reboulet, Institut Camille Jordan, Université Claude Bernard Lyon 1, 43 Bd du 11 Novembre 1918, 69100 Villeurbanne, France}\email{reboulet@math.univ-lyon1.fr, rebouletremimath@gmail.com}
\begin{document}

\begin{abstract}Geometric Invariant Theory (GIT) produces quotients of algebraic varieties by reductive groups. If the variety is projective, this quotient depends on a choice of polarisation; by work of Dolgachev--Hu and Thaddeus, it is known that two quotients of the same variety using different polarisations are related by birational transformations. Only finitely many birational varieties arise in this way: variation of GIT fails to capture the entirety of the birational geometry of  GIT quotients. We construct a space parametrising all possible GIT quotients of all birational models of the variety in a simple and natural way, which captures the entirety of the birational geometry of GIT quotients in a precise sense. It yields in particular a compactification of a birational analogue of the set of stable orbits of the variety.
\end{abstract}

\maketitle

\section{Introduction}

Let $X$ be a complex projective variety with an action of a reductive group $G$. Given an ample, $G$-linearised  line bundle $L$ on $X$ (namely a polarisation), Geometric Invariant Theory (GIT), introduced by Mumford \cite{gitbook}, yields a projective quotient variety $X\sslash_L G$ parametrising polystable orbits in $X$.

The construction of the GIT quotient relies heavily on the polarisation $L$, and it is natural to investigate what happens when the polarisation $L$ varies. Fundamental work of Dolgachev--Hu \cite{vgit1} and Thaddeus \cite{vgit2} on variation of GIT (VGIT) shows that all GIT quotients of $X$ by $G$ are birational (under the assumption of the existence of a stable point). This structure induces a finite wall-and-chamber structure on the $G$-ample cone of $X$, where all quotients with respect to polarisations in a given chamber are isomorphic; walls divide the chambers, and when crossing a wall the GIT quotients undergo certain birational transformations which are now called Thaddeus flips.

GIT is  used both to construct moduli spaces in algebraic geometry and as an analogy to understand the structure of algebro-geometric moduli spaces constructed using other techniques. The analogue of VGIT in moduli theory studies the birational geometry of moduli spaces, notably by varying stability conditions, and this is by now a prominent topic in algebraic geometry. The theory of VGIT by nature produces only finitely many possible birational models of the GIT quotient, hence cannot in any sense fully capture the birational geometry of GIT quotients. This deficiency is what we address in the current work.

Our approach  borrows from the theory of Zariski--Riemann spaces. Given $X$ as above, its Zariski--Riemann space is the projective limit $$\underline{X}:=\varprojlim \,(\pi_Y:Y\to X)$$
taken over the set of all birational morphisms $\pi_Y:Y\to X$, ordered by dictating that $\pi_Y:Y\to X$ is smaller than $\pi_Z:Z\to X$ if there exists a birational morphism $\nu:Z\to Y$ such that $\pi_Z$ factors through $\pi_Y$. The space $\underline{X}$, viewed naturally as a locally ringed space, then completely determines the birational geometry of $X$: birational varieties have isomorphic Zariski--Riemann spaces tautologically, while conversely the isomorphism type of  $\underline X$ as a locally ringed space determines the birational equivalence class of $X$, as explained in Section \ref{RZspaces}.

Zariski--Riemann spaces were originally introduced by Zariski \cite{zariskisurf} in his study of resolution of singularities for surfaces; to this day, they see use in various areas of complex geometry, for example in the minimal model programme \cite{shokurovflips}, complex dynamics \cite{dangfavrespectral}, K-stability \cite{bigding, trusiani-ytd} and non-Archimedean pluripotential theory \cite{bjtrivval}. We refer to Vaqui\'e for an introduction to Zariski--Riemann spaces and related ideas \cite{vaquie}.

In our setting, we rather consider as objects the set of \textit{GIT quotients over $X$}, i.e. GIT quotients
$$Y\sslash_{L_Y}G,$$
where $Y$ admits a $G$-equivariant birational morphism to $X$, and $L_Y$ is an ample, $G$-linearised line bundle on $Y$ admitting a stable point (the latter assumption is a running assumption in most papers on GIT, for example \cite{kirwandesing}). Our first main result shows that all such quotients fit naturally into a projective system whose limit is isomorphic to the Zariski--Riemann space of any  GIT quotient of any birational model of $X$ (in particular, of $X$ itself).

\bigskip\noindent\textbf{Theorem A.} GIT quotients over $X$ are mutually birational, and form a projective system. The projective limit
$$V^G\underline{X}:=\varprojlim (Y\sslash_{L_Y} G)$$
is isomorphic, as a locally ringed space, to the Zariski--Riemann space $\underline{Y\sslash_{L_Y} G}$ of any GIT quotient over $X$.

\bigskip The space $V^G\underline{X}$ can therefore be thought of as a \textit{universal VGIT space} for $X$ and $G$. The second half of this result implies that $V^G\underline{X}$ \emph{completely determines} the birational geometry of GIT quotients of (and over) $X$, as we had promised. We highlight the difference in our approach in comparison to VGIT: in VGIT, one varies only the line bundle, producing \emph{finitely many} birational models of $X\sslash_{L}G$. We instead also allow $X$ itself to vary birationally, and by allowing all GIT quotients of all equivariant birational models, we obtain \emph{all possible} birational information of the GIT quotient $X\sslash_{L}G$. 

Once the dimension of $X\sslash_{L}G$ is at least two, the projective limit $V^G\underline{X}$ is not a scheme: by Theorem A, it is isomorphic to the Zariski--Riemann space $\underline{Y\sslash_{L_Y} G}$ of any GIT quotient $Y\sslash_{L_Y} G$ over $X$, which is itself not a scheme in dimension at least two \cite[Corollary 5.3]{olberding}. In particular, $V^G\underline{X}$ cannot be isomorphic to any specific GIT quotient of $X$. As a classical example, if one considers the moduli space of hypersurfaces of degree $d$ in projective space $\pr^n$ given by the GIT quotient $\pr(H^0(\pr^n,\scO(d)))\git_{\scO(1)}\GL(n+1)$, this quotient has dimension at least two once $n=2$ and $d\geq 4$ or $n\geq 3$ and $d\geq 3$, as follows from \cite[Corollary 5.30]{mukai}. Thus the locally ringed space produced by Theorem A is not a scheme for these moduli spaces, involving infinitely many birational models, whereas in general VGIT involves only finitely many. Conversely, in any situation in which $X\sslash_{L}G$ has dimension one, the projective system is trivial (there being no nontrivial birational morphisms) and the projective limit is hence a variety.

By construction,  $V^G\underline{X}$ admits canonical surjective morphisms to all possible GIT quotients of all $G$-equivariant birational models of $X$. In the case $G$ is trivial, we recover the Zariski--Riemann space of $X$ from our projective limit; in general, our construction can be thought of as producing a $G$-equivariant analogue of the Zariski--Riemann space of $X$. Much as with Zariski--Riemann spaces themselves, describing $V^G\underline{X}$ explicitly is quite challenging: this is already the case for the Zariski--Riemann space of $\pr^n$.

We emphasise here that our notion of a projective system requires the system to be \emph{filtered}: for any two GIT quotients over $X$, we construct a third which admits birational morphisms to both. This property is what forces us to consider birational models of $X$, and is  one of the key new aspects of our approach. We rely on important technical results of Kirwan \cite{kirwandesing} and  generalisations thereof due to Reichstein \cite{reichstein} to establish this property. By contrast with a \emph{weaker} definition of a projective system (dropping the filtered requirement) one can form a system of GIT quotients of a fixed variety by the theory of VGIT, and can hence take the limit of the resulting system, sometimes called the \emph{limit quotient} of $X$; the limit quotient may be reducible. In fact by a beautiful result of Hu \cite[Theorem 3.8]{yihuchow}, one irreducible component of the limit quotient is homeomorphic (in the analytic topology) to the \emph{Chow quotient} of $X$, introduced by Kapranov \cite{kapranov, ksz}. Since in Hu's construction the projective limit is taken over a finite system, the end result is that the limit quotient is actually a (possibly reducible) \emph{variety}, whereas by contrast we merely obtain a \emph{locally ringed space}. The advantage of our construction is that by allowing birational models of $X$ as well, we obtain a space which completely determines the birational geometry of GIT quotients of $X$, at the expense of leaving the category of varieties. By general theory, we note that $V^G\underline{X}$  also admits  canonical surjective morphisms to all limit quotients of all birational models of $X$.

Our original motivation for investigating the birational behaviour of GIT quotient was to understand structures of moduli spaces parametrising (suitably stable) projective varieties with big line bundles (following \cite{bigding}), where we expect such moduli spaces to have similar Zariski--Riemann type structures. In particular Theorem A is a model analogue of the sort of structure that we expect such moduli spaces to admit, and we hope to come back to this in future work.

\bigskip

GIT quotients are frequently described as compactifications of the space of stable orbits, and it is desirable to have an analogous description of $V^G\underline{X}$. In order to produce such a description, we next analyse the set of $\bbc$-points of the space $V^G\underline{X}$. A way to capture such points is to take its analytification, in the sense of Serre (in the category of projective limits of varieties), i.e. to consider the projective limit
$$V^G\underline{X}\an:=\varprojlim\,(Y\sslash_{L_Y}G)\an,$$ which by the general theory of GIT and projective systems is a compact, Hausdorff topological space.
We show that it is a compactification the set of \textit{infinitesimally stable points} over $X$, which are points $\underline x\in V^G\underline{X}\an$ such that their realisations on GIT quotients  $Y\sslash_{L_Y}G$ over $X$ are stable for a ``sufficiently large'' collection of GIT quotients over $X$; technically we ask that the set of such GIT quotients over $X$ is cofinal in the projective system. The name is chosen by analogy with the perspective of viewing blow-ups as adding ``infinitesimal directions'' and the fact that birational morphisms are blow-ups.

\bigskip\noindent\textbf{Theorem B.} The set of infinitesimally stable points is an open, dense subset of the compact space $V^G\underline{X}\an$. 

\bigskip Loosely, any stable orbits that could potentially be missing (in a birational sense) from the usual quotient $X\sslash_L G$ because they are ``hidden'' in an infinitesimal direction are captured by the universal VGIT quotient $V^G\underline{X}$.

\subsection*{Wall-crossing} The main result of VGIT is that the wall-and-chamber decomposition of the $G$-ample cone of a projective variety is actually finite (see Ressayre for refinements \cite{ressayre}). The $G$-N\'eron--Severi groups (tensored with $\Q$, say) of $G$-equivariant birational models of $X$ form a projective system through pushing-forward, and hence one obtains a vector space given by taking the projective limit of these vector spaces (elements of this projective system are called $b$-divisors  \cite{shokurovflips}). With the projective limit topology (also called the initial topology), however, there is no hope to obtain a locally finite wall-and-chamber decomposition,  for the following reason (assuming the dimension of the GIT quotients is at least two, so that its birational geometry is nontrivial). 

Given an ample $G$-linearised line bundle  $L$ on $X$, take an open set $U$ inside the projective limit of $G$-N\'eron--Severi groups with respect to the projective limit topology. For \emph{any} $G$-equivariant model $\pi_Y: Y \to X$, for $E$ relatively ample and $G$-linearised the line bundle $L-tE$ is both an ample $G$-linearised line bundle and inside $U$ for $t$ sufficiently small. Thus from the techniques used to prove Theorem A, there will be infinitely many birational models in any neighbourhood of $L$. In other uses of Zariski--Riemann spaces---notably the work of Dang--Favre \cite{dangfavrespectral} on complex dynamics---other topologies are introduced, and it would be interesting to see if there is a natural topology for which one \emph{can} obtain a locally finite wall-and-chamber decomposition. This would follow if a basis for the topology could be taken to be given by line bundles which lie on a fixed birational model of $X$;  this, however, is an extremely unnatural condition for a topology on this projective limit to satisfy and it is unlikely that any interesting topology actually satisfies this property.

\subsection*{Analogues in moduli theory} As already mentioned, these days GIT is used as much as a motivational philosophy as as a practical tool to construct moduli spaces. Perhaps most prominent of the analogues of VGIT is the theory of Bridgeland stability conditions on triangulated categories \cite{TB}. The analogue of our results in this setting would be roughly as follows; we do not aim to make any precise claims, but instead just loosely describe possible analogues in this setting, focusing on the case that the triangulated category is the (bounded) derived category $\mathcal D^b\Coh(Y)$ of coherent sheaves on a projective variety $Y$, which is the most important one in algebraic geometry.

A stability condition on $\mathcal D^b\Coh(Y)$ requires in particular a choice of abelian subcategory $\scA\subset \mathcal D^b\Coh(Y)$ and a central charge $Z$ on $\scA$, which determines a notion of semistability, polystability and stability for objects in $\scA$. One should expect to obtain a moduli space $\M_{\scA,Z}$ of polystable objects in $\scA$ with respect to the central charge $Z$ (and one frequently does---see for example \cite{toda}). In particular one should think of $\scA$---which is chosen to lie in the triangulated category  $\mathcal D^b\Coh(Y)$ itself---as the analogue of the space $X$ one is taking a  GIT quotient of, and $Z$ as a choice of ample $G$-linearised line bundle; in particular, it is frequently emphasised that stability conditions are best thought of as polarisations on triangulated categories, in that they determine moduli spaces (though perhaps a K\"ahler class rather than a polarisation is a more accurate analogy). The stability manifold of a fixed triangulated category then admits a locally finite wall-and-chamber decomposition, as an analogue of the fundamental results of VGIT (local finiteness in this setting requires a strengthening of the original definition \cite{KS}).

Thus from our perspective, while one cannot expect the birational geometry of $\M_{\scA,Z}$ to be completely determined by varying the stability condition on $\mathcal D^b\Coh(Y)$, one might expect to be able to do so by ``birationally'' varying the analogue of $X$: the triangulated category itself. Enticingly, in this situation there are natural geometric variations of the triangulated category---by considering the derived category of coherent sheaves of its birational models---and so one can ask if the birational geometry of $\M_{\scA,Z}$ is completely determined by moduli spaces of polystable objects in $\mathcal D^b\Coh(W)$ with respect to all possible stability conditions and all possible birational models $W$ of $Y$. This is likely rather too optimistic a suggestion for various reasons, and should not be taken too literally, but would be an analogue of our Theorem A in this setting and would fit in well with the vast literature on the interplay between birational geometry and derived categories of projective varieties (an early reference for which is Bondal--Orlov \cite{BO}).

\subsection*{Symplectic reduction} A foundational result in GIT due to Kempf--Ness and Kirwan relates GIT to symplectic geometry \cite{kempf-ness, kirwan-thesis}, by showing that the (analytification of the) GIT quotient of $X$ with respect to $L$ is homeomorphic to the symplectic quotient of $X\an$ by a maximal compact subgroup $K\subset G$ with respect to a $K$-invariant K\"ahler metric $\omega \in c_1(L)$; denoting by $\mu_L: X\an \to \mfk^*$ the associated moment map, there is a natural homeomorphism $$(X \sslash_L G){\an} \cong \mu_L\mi(0)/K,$$where we have taken the analytification and hence have endowed $(X \sslash_L G){\an}$ with the complex topology. \emph{A priori} there no is projective system with elements of the form $\mu_{L_Y}^{-1}(0)$ ranging over $G$-equivariant birational models $(Y,L_Y)$ of $X$ along with associated moment maps (similarly semistable loci on birational models should not be expected to form a projective system), making it essentially impossible to phrase an analogue of this result in our setting. However, trivially by Theorem A the symplectic quotients $ \mu_{L_Y}\mi(0)/K$ \emph{do} form a projective system with the same hypotheses assumed there, producing by the general theory of projective systems a homeomorphism $$V^G\underline{X}{\an}=\varprojlim (Y\sslash_{L_Y} G)\an \cong \varprojlim \mu_{L_Y}\mi(0)/K,$$ which seems to be the best that one can hope for.

\bigskip\noindent\textbf{Organisation of the paper:} 
In Section \ref{sect_preliminaries}, we recall the basics of GIT, and collect various results concerning projective limits of topological spaces and Zariski--Riemann spaces. Section \ref{sect_main} then constructs the universal VGIT space and is where we prove Theorem A, which is split into Proposition \ref{prop_allbir}, Theorem \ref{thm_uvgit}, and Theorem \ref{thm_zr}. We end in Section \ref{sect_compact} by defining infinitesimally stable points and proving the compactification result Theorem B, as Theorem \ref{thm_compact}.

\bigskip\noindent\textbf{Notation:} We work over an algebraically closed field of characteristic zero; we choose $\C$ for aesthetic reasons.  A compact topological space is one that is both  quasicompact and Hausdorff. A variety is an integral scheme that is separated and finite type over $\C$. We thus use the Zariski topology, and points will therefore mean scheme-theoretic points. The exception to this is the material on analytifications, where we use the complex topology and where this change-of-topology will be made clear; a $\C$-point in $X$ thus corresponds precisely to a point  in $X\an$.

\bigskip\noindent\textbf{Acknowledgements:} We thank Frances Kirwan for a discussion that strongly influenced the direction of our work, and in addition thank S\'ebastien Boucksom for a helpful conversation which prompted us to consider cofinality of GIT quotients in the projective system defining the Zariski--Riemann space. We also thank the referee for their helpful comments. RD was funded by a Royal Society University Research Fellowship, while RR was supported by a grant from the Knut and Alice Wallenberg foundation.

\section{Preliminaries.}\label{sect_preliminaries}

\subsection{Geometric invariant theory.}\label{sec:GIT}  We begin by recalling the basic results of GIT, for which the book \cite{gitbook} is the foundational reference. A more introductory account is given by Hoskins \cite{hoskins}. Proofs of all of the following claims can be found in these references (though the terminology of \cite{gitbook} is now non-standard and disagrees with what is used below).

We consider a projective variety $X$ along with an ample line bundle $L$; in addition we consider a reductive group $G$ which acts on $(X,L)$.

\begin{definition} The \emph{GIT quotient} of $(X,L)$ by $G$ is defined to be $$X\sslash_L G = \Proj \oplus_{k\geq 0} H^0(X,kL)^G,$$ where $H^0(X,kL)^G$ denotes $G$-invariant global sections. \end{definition}

The general theory of GIT implies that $X\sslash_L G$ is also a projective variety; in particular the graded ring $\oplus_{k\geq 0} H^0(X,kL)^G$ is finitely generated, and integrity and separatedness are preserved under GIT quotients. The geometric meaning of the GIT quotient is elucidated through various notions of \emph{stability} of orbits.

\begin{definition}
We say that a $\C$-point $p \in X$ (or equally the orbit $G.x$) is
\begin{enumerate}[(i)]
\item \emph{semistable} if there is a section $s\in H^0(X,kL)^G$ for some $k>0$ such that $s(x)\neq 0$;
\item \emph{polystable} if the orbit $G.x$ is closed in the semistable locus $X^{ss}$;
\item \emph{stable} if $x$ is polystable and the stabiliser $G.x$ is finite;
\item \emph{unstable} otherwise.
\end{enumerate} 
\end{definition}

Clearly stability implies polystability which in turn implies semistability. In the above, it is also straightforward to see that the semistable locus is Zariski open in $X$, since the condition that $s\in H^0(X,L^k)^G$ is non-vanishing at a point is a Zariski open condition. The same is true of the stable locus:

\begin{lemma} Both the semistable locus $X^{ss}$ and the stable locus $X^s$ are Zariski open subsets of $X$. \end{lemma}

In principle both the semistable and the stable loci could be empty; when the semistable locus is empty the GIT quotient is patently trivial. When the semistable locus is \emph{not} empty there is a morphism $$X^{ss}\to X\sslash_L G;$$ beyond this one only obtains a rational map $X\dashrightarrow X\sslash_L G$ which is undefined on the complement of the semistable locus (i.e. on the unstable locus). Again by the general theory of GIT the closure of each semistable orbit contains a (unique) polystable orbit, so the existence of a semistable orbit implies the existence of a polystable orbit. The condition that the stable locus not vanish---which is a running assumption throughout the present work---is thus a non-degeneracy assumption which is essentially standard in GIT (assumed for example in \cite{kirwandesing}). The fact that the closure of each semistable contains a \emph{unique} polystable orbit implies the following:

\begin{lemma} The $\C$-points of the  GIT quotient $X\sslash_L G$ are in bijection with polystable orbits.\end{lemma}

Thus GIT quotients parametrise polystable orbits. This result is sometimes phrased in terms of $S$-equivalence, namely it is often (equivalently) stated that GIT quotients parametrise semistable orbits up to $S$-equivalence, where two semistable orbits are $S$\emph{-equivalent} if their closures intersect (thus their associated polystable orbits agree and hence they induce the same point in the GIT quotient). On the stable locus $X^s$, the  GIT quotient  $X^s\to X^s\sslash_L G$ is actually a \emph{geometric} quotient, which in particular implies that the fibres of the quotient morphism are $G$-orbits.

Locally, GIT quotients can be constructed as follows. Let $x \in X$ be a semistable point, with $s(x) \neq 0$ for some $s\in H^0(X,kL)^G$. Then since $L$ is ample, the complement $X - V(s)$ is affine. 

\begin{lemma}\label{zariskiopenaffine} The affine GIT quotient $$(X - V(s))\sslash G := \Spec( \C[X - V(s)]^G)$$ is an affine chart of $X\sslash_L G$, and is in particular Zariski dense if it is nonempty (for example when $x$ is stable).

\end{lemma}

Note that affine GIT quotients are independent of polarisation.

\subsection{Projective limits}
We move on to projective limits, and first simply collect a few definitions and facts. 

\begin{definition}Let $\cC$ be a category. We define a \textit{projective system} of objects in $\cC$ to be the data of:
\begin{enumerate}[(i)]
\item a (right-)filtered poset $(I,\geq)$, i.e. a poset such that, for each tuple $i,j\in I$, there exists $k\in I$ with $k\geq i$ and $k\geq j$;
\item for each $i\in I$, an object $X_i$ of $\cC$;
\item for each ordered tuple $i\leq j\in I$, a morphism $f_{ij}:X_j\to X_i$ in $\cC$, such that if $i\leq j \leq k$, then $f_{ik}=f_{ij}\circ f_{jk}$.
\end{enumerate}
Note that require our projective systems be \textit{filtered}, which is nonstandard.
\end{definition}

\begin{definition}We define, as a set, the \textit{projective limit} of a projective system as above to be
$$\varprojlim_I X_i :=\left\{\underline{a}\in \prod_I X_i,\,a_i=f_{ij}(a_j)\,\forall i,j\in I\right\}.$$
It canonically comes with projection maps
$$\pi_i:\varprojlim_I X_i\to X_i$$
for all $i$, sending $\underline{a}$ to $a_i$; the $a_i$ are called \emph{realisations} of $\underline a$. If the category $\cC$ is that of topological spaces, the limit is itself a topological space, endowed with the coarsest topology making all the maps $\pi_i$ continuous; this topology is called the \emph{projective limit topology} or the \emph{initial toplogy}.
\end{definition}

\begin{definition}Let $(J,\geq)$ be a sub-poset of a poset $(I,\geq)$. We say that $J$ is \textit{cofinal} in $I$ if, for all $i\in I$, there exists $j\in J$ with $j\geq i$.
\end{definition}

\begin{proposition}{\cite[Chapitre 3, \textsection 7, Proposition 3]{bourbakiens}}\label{seb} If $(X_i,f_{ij})_I$ be a projective system in some category $\cC$, indexed by a poset $(I,\geq)$, and $J$ is a cofinal sub-poset of $I$, then there is a canonical identification of projective limits
$$\varprojlim_{j\in J} X_j\simeq \varprojlim_{i\in I}X_i.$$
\end{proposition}

We will require the following technical result on dense open subsets in projective systems of topological spaces.

\begin{proposition}\label{prop_projlimdensity}Let $(X_i,f_{ij})_{i,j\in I}$ be a projective system of compact topological spaces with surjective maps. Let, for some $i\in I$, $U_i$ be a dense open subset of $X_i$. Then $\pi_i\mi U_i$ is a dense open subset in $\varprojlim_I X_i$ with respect to the projective limit topology.
\end{proposition}
\begin{proof}
Let $j\in I$, and fix $x_j\in X_j$. Since open sets of the form $\pi_j\mi(V_j)$ form a basis for the projective limit topology, by \cite[Chapitre 1, \textsection 4, Proposition 9]{bourbakitop14}, it suffices to show that $\pi_j\mi V_j \cap \pi_i\mi U_i$ is nonempty. 

Consider $k\leq i,j$. Note that $X_i$ and $X_k$ are compact and $f_{ik}$ is surjective and continuous, and is hence an open map. Since density is preserved by open maps, it follows that $f_{ik}\mi U_i$ is dense in $X_k$. Furthermore, since $f_{jk}$ is surjective, $f_{jk}\mi V_j$ is a nonempty open set. Combining these two facts, the open set $(f_{ik}\mi U_i)\cap (f_{jk}\mi V_j)\subset X_k$ is nonempty.

Since the maps $f_{ij}:U_j\to U_i$ are surjective, the projection maps $\pi_i:\varprojlim_I X_i\to X_i$ are also surjective by \cite[Chapitre 3, \textsection 7, Proposition 5]{bourbakiens}. Thus, $$\pi_k\mi((f_{ik}\mi U_i)\cap (f_{jk}\mi V_j))=(\pi_k\mi(f_{ik}\mi U_i))\cap(\pi_k\mi(f_{jk}\mi V_j))$$
is nonempty. Thus, to prove the proposition, it now suffices to notice that $$\pi_k\mi(f_{ik}\mi U_i)\subset \pi_i\mi U_i$$ and that $$\pi_k\mi(f_{jk}\mi V_j)\subset \pi_j\mi V_j,$$ which is trivial by definition of the maps $f_{ij}$ and $\pi_i$.
\end{proof}

\subsection{Zariski--Riemann spaces}\label{RZspaces}

We now let $X$ be a compact, complex projective variety.

\begin{definition}
The \emph{Zariski--Riemann space} $\underline X$ of $X$ is the  projective limit
$$\underline{X}:=\varprojlim\,(\pi_Y:Y\to X),$$
where the limit is taken over the set of birational morphisms $\pi_Y:Y\to X$ (which we will call \textit{birational models} of $X$), with $Y$ a complex projective variety, and one says that $\pi_Z$ \emph{dominates} $\pi_Y$ if there exist a birational morphism $\nu$ with $\pi_Z=\pi_Y\circ\nu$.
\end{definition}

Beyond the case where $\dim X=1$, where $\underline X$ is isomorphic to the (unique) smooth model of $X$, $\underline X$ is not a scheme in general, as there exist infinitely many birational models of $X$ as above -- in particular, a point in $\underline X$ cannot admit an affine neighbourhood in the limit Zariski topology. It is however known to be a locally ringed space \cite[Proposition 4.1.10]{katofujiBook} (the fact that the Zariski--Riemann itself is a locally ringed space is much more classical). Zariski--Riemann spaces are typically defined in terms of valuations; this perspective will only briefly be used in this Section and will play no further role beyond this. We will always consider the Zariski--Riemann space as a locally ringed space.

We now justify our claim in the introduction, that Zariski--Riemann spaces determine the birational equivalence class of a projective variety. To make this precise, consider the set consisting of locally ringed spaces $(S,O_S)$ which are the  Zariski--Riemann space of some projective variety, where we emphasise that the information of \emph{which} variety $S$ is the Zariski--Riemann space of is \emph{not} part of the information. Consider also the set consisting of function fields of projective varieties; we recall that two projective varieties are birational if and only if their function fields agree, and so the set of function fields precisely corresponds to the set of projective varieties up to birational equivalence. We show that these two sets are naturally in bijection. 

\begin{proposition}\label{birationalRZ}
The Zariski--Riemann spaces of two projective varieties $X$ and $Y$ are isomorphic as locally ringed spaces if and only if $X$ and $Y$ are birational.
\end{proposition}

\begin{proof}

Assume that $X$ and $Y$ are two birational varieties. Let $X'$ be a birational model for $X$, and $Y'$ be a birational model for $Y$. Then, there exists a model $Z'$ dominating $X'$ and $Y'$. In other words, the projective system of projective varieties that dominate $X$ and $Y$ is cofinal in the projective systems of birational models of $X$ and $Y$, meaning that the projective limits of all three systems are isomorphic by Proposition \ref{seb}.

Conversely, let $(S,O_S)$ be a locally ringed space which is the Zariski--Riemann space of two varieties $X$ and $Y$. Let $s\in S$. By \cite[Theorem 3.2.5]{localisation}, $s$ is identified with a valuation $\nu_X$ on the function field of $X$, which is trivial on $\bbc$. Likewise, it also corresponds to such a valuation $\nu_Y$ on $\bbc(Y)$. To those valuations are associated valuation rings $\cO_{\nu_X}$, $\cO_{\nu_Y}$. One now has isomorphisms
$$\cO_{\nu_X}\simeq\cO_{s,S}\simeq\cO_{\nu_Y},$$
given by \cite[Lemma 3.2.3]{localisation}. We can pick $s$ such that $\cO_{s,S}$ is integral. By \cite[Chapitre 6, \textsection 2, Théorème 1]{bourbakicommalg}, it follows that the fraction field of $\cO_{\nu_X}$ is $\bbc(X)$, and that of $\cO_{\nu_Y}$ is $\bbc(Y)$. That is, the local ring at $s\in S$ determines the function fields of both $X$ and $Y$, which implies that  the function fields of $X$ and $Y$ are actually isomorphic, concluding the proof.\end{proof}

Thus the information of the Zariski--Riemann space of a projective variety is completely equivalent to its birational equivalence class: the Zariski--Riemann space determines the birational geometry of the variety.

\section{Universal variation of GIT}\label{sect_main}

\subsection{Technical results on relative GIT}

We first recall the following result due to Reichstein which compares stable and semistable loci of two projective varieties related by a birational morphism; the result is originally due to Kirwan in the smooth case  \cite{kirwandesing}. We also note that a quasiprojective version, which we shall not need here, appears in work of Hu \cite[Theorem 3.13]{yihurelgit}.

\begin{theorem}{\cite[Theorem 2.1]{reichstein}}\label{thm_compstable}
Let $\pi:Z\to Y$ be a birational morphism of projective varieties. Let $L_Y$ be an ample $G$-linearised line bundle on $Y$, and let $L_Z$ be an ample $G$-linearised line bundle on $Z$. Let $L_{t}:=\pi^*L_Y-tL_Z$. Then, for all $t$ sufficiently small,
\begin{enumerate}[(i)]
\item $Z^{ss}\subset\pi\mi(Y^{ss})$;
\item $\pi\mi(Y^s)\subset Z^{s}$,
\end{enumerate}
where $Z^{ss}$ and $Z^s$ denote the semistable and stable loci with respect to $L_t$ respectively.
\end{theorem}
\begin{example}\label{ex_compstable}
By \cite[Proposition 7.10]{hartshorne}, if $\pi$ is a blow-up with exceptional divisor $E$, then one can replace $L_t$ with $\pi^*L_Y-tE$ in the statement of Theorem \ref{ex_compstable}.
\end{example}

We will also crucially rely on a result of Kirwan similarly around GIT quotients of birational models \cite[Lemma 3.11]{kirwandesing}. Although Kirwan's result is stated in the smooth setting, her proof goes through verbatim in the singular setting using the above result of Reichstein: 

\begin{lemma}\label{lem_kirwan}Let $Y$ be a  projective variety on which a reductive group $G$ acts, endowed with a $G$-linearised ample line bundle $L_Y$. Let $V\subset Y$ be  fixed by the action of $G$, and consider the blow-up
$$\mu_V:\Bl_V Y\to Y$$
with exceptional divisor $E_V$. Then, for all $t$ sufficiently small, there is an isomorphism
$$(\Bl_V Y)\sslash_{L-tE_V}G\simeq \Bl_{V\sslash_{L_Y}G}(Y\sslash_{L_Y}G),$$
where  $V\sslash_{L_Y} G$ denotes the image of $V$ under the quotient morphism $Y^{ss}\to Y\sslash_{L_Y}G$.
\end{lemma}

\subsection{Construction of the projective limit}

We now fix a projective variety $X$ together with an action of a reductive group $G$  on $X$.

\begin{definition}A \textit{GIT quotient (by $G$) over $X$} is the data of a GIT quotient $Y\sslash_{L_Y}G$, where:
\begin{enumerate}[(i)]
    \item there is a $G$-equivariant, birational morphism $\pi_Y:Y\to X$, with $Y$ a  projective variety on which  a reductive group $G$ acts;
    \item $L_Y$ is a $G$-linearised ample line bundle on $Y$;
    \item the stable locus on $Y$ for the $G$-action with respect to the polarisation $L_Y$ is nonempty.
\end{enumerate}
\end{definition}

We begin with the observation that all GIT quotients over $X$ are birational.
\begin{proposition}\label{prop_allbir}Let $Y\sslash_{L_Y}G$ and $Z\sslash_{L_Z}G$ be GIT quotients over $X$. Then, there exists a  birational map $\mu:Y\sslash_{L_Y}G\dashrightarrow Z\sslash_{L_Z}G$.
\end{proposition}
\begin{proof}By assumption there exist a stable point $x_Y\in Y$ and a section $s_{k_YY}\in H^0(Y,k_Y L_Y)$ with $s_{Y}(x_Y)\neq 0$ for some positive integer $k_Y$, as well as a stable point $x_Z\in Z$ and a section $s_{Z}\in H^0(Z,k_ZL_Z)$ with $s_{Z}(x_Z)\neq 0$ for some positive integer $k_Z$. Let $V_Y$, $V_Z$ denote the vanishing loci of $s_Y$ and $s_Z$ respectively. Then, because $x_Y$ is furthermore stable, the affine GIT quotient $(Y-V_Y)\sslash G$ embeds as a Zariski-dense open subset of $Y\sslash_{L_Y}G$ by Lemma \ref{zariskiopenaffine}, and likewise for $Z$. 

On the other hand, both $Y$ and $Z$ are $G$-equivariantly birational to $X$, hence to each other, so that one may find a Zariski open subset $U$ which embeds $G$-equivariantly into both $Y-V_Y$ and $Z-V_Z$; Zariski openness of the stable locus implies $U$ also contains stable points. Then, the (affine) GIT quotient $U\sslash G$ embeds as a Zariski-dense open subset of both $(Y-V_Y)\sslash G$ and $(Z-V_Z)\sslash G$, which are respectively dense in $Y\sslash_{L_Y}G$ and $Z\sslash_{L_Z}G$, proving the desired result.
\end{proof}

\begin{definition}We define a partial order on the set of GIT quotients over $X$, whereby $Y\sslash_{L_Y}G \geq Z\sslash_{L_{Z}}G$ if and only if there exists a birational \textit{morphism} $\pi:Y\sslash_{L_Y} G\to Z\sslash_{L_Z} G$.
\end{definition}
We now show that, endowed with this partial order, the set of GIT quotients over $X$ defines a projective system.

\begin{theorem}\label{thm_uvgit}The set of GIT quotients over $X$, together with the partial order  $\geq$ defined above, is a projective system.
\end{theorem}
\begin{proof}
We must show that, given two GIT quotients $Y\sslash_{L_Y}G, Z\sslash_{L_{Z}}G$ over $X$, there exists a GIT quotient $W\sslash_{L_W}G$ over $X$, together with birational morphisms $$W\sslash_{L_W}G\to Y\sslash_{L_Y}G,$$ $$ W\sslash_{L_W}G\to Z\sslash_{L_Z}G.$$
Because $Y\sslash_{L_Y}G$ and $Z\sslash_{L_Z} G$ are birational, one can find a subvariety $V\subset Y\sslash_{L_Y}G$ such that there exists a birational morphism
$$\nu_V:\Bl_V(Y\sslash_{L_Y}G)\to Z\sslash_{L_Z} G.$$
Furthermore, by definition of the blow-up, there also exists a birational morphism
$$\mu_V:\Bl_V(Y\sslash_{L_Y}G)\to Y\sslash_{L_Y} G.$$
Notice now that if we can choose $V$ to be of the form $V=U\sslash_{L_Y}G$ for some $G$-subvariety $U\subset Y$, then our claim would be proven. Indeed, by Lemma \ref{lem_kirwan} above, we would then have for $t$ small enough that
$$\Bl_V(Y\sslash_{L_Y}G)=\Bl_{U\sslash_{L_Y}G}(Y\sslash_{L_Y}G)\simeq (\Bl_U Y)\sslash_{L_Y-tE_U}G,$$
with $E_U$ the exceptional divisor of the blow-up $\mu_U:\Bl_UY\to Y$. Hence the morphisms $\mu_V$, $\nu_V$ as above would be morphisms from a GIT quotient over $X$, since $\Bl_U Y$ is birational to $X$, $L_Y-tE_U$ is ample for $t$ small enough, and admits a stable point by Theorem \ref{thm_compstable} $(ii)$ (since $Y$ admits one); this would prove the claim of the Theorem.

We now construct the desired blow-up. Let $V$ be such that there are morphisms $\mu_V$, $\nu_V$ as above. Then, let $\pi_Y:Y^{ss}\to Y\sslash_{L_Y}G$ be the quotient morphism, and let $U$ be the closure of $\pi_Y\mi(V)$ in $Y^{ss}$. It then follows that $\pi_Y(U^{ss})=V$, and the result is proven.
\end{proof}

\begin{definition}We define the \textit{universal VGIT space} $V^G\underline{X}$ to be the projective limit of the system given by GIT quotients over $X$, viewed as a locally ringed space.
\end{definition}

We note the following result, mentioned in the introduction:

\begin{proposition}\label{prop_surjmaps}For any GIT quotient $Y\sslash_{L_Y}G$ over $X$, the projection map
$$V^G\underline{X}\to Y\sslash_{L_Y}G$$
is surjective. 
\end{proposition}
\begin{proof}
By \cite[Chapitre 1, \textsection 9, Proposition 8]{bourbakitop14}, the image of the projection map $V^G\underline{X}\to Y\sslash_{L_Y}G$ equals the intersection
$$\bigcap_{Z\sslash_{L_Z}G\geq Y\sslash_{L_Y}G}\pi(Z\sslash_{L_Z}G),$$
where $\pi$ denotes the birational morphism $Z\sslash_{L_Z}G\to Y\sslash_{L_Y}G$. But since birational morphisms are surjective, its image is therefore all of $Y\sslash_{L_Y}G$, concluding the proof.
\end{proof}

\subsection{Universality}

We next prove a cofinality result, which is important in justifying our claims concerning the birational geometry of GIT quotients. 
\begin{theorem}\label{thm_zr}
The space $V^G\underline{X}$ is isomorphic, as a locally ringed space, to the Zariski--Riemann space $\underline{Y\sslash_{L_Y} G}$, where $Y\sslash_{L_Y} G$ is any GIT quotient over $X$.
\end{theorem}
\begin{proof}
By Proposition \ref{seb} it suffices to show that the projective system given by GIT quotients over $X$ is cofinal as a projective sub-system of the projective system given by all birational models of $X\sslash_{L_X} G$. In other words, we must show that, if $Y\sslash_{L_Y}G$ is a GIT quotient over $X$, and $Z$ admits a birational morphism to $Y\sslash_{L_Y}G$, then there exists a GIT quotient over $X$ of the form $W\sslash_{L_W} G$, with a birational morphism to $Z$. From the proof of Theorem \ref{thm_uvgit}, it suffices to find a subvariety $V$ of $Y\sslash_{L_Y}G$ such that there exists a morphism
$$\Bl_V(Y\sslash_{L_Y}G)\to Z;$$ taking any such $V$ concludes the proof of the result. \end{proof}

This result implies that the universal VGIT space $V^G\underline{X}$ \emph{completely determines} the birational equivalance class of $X\sslash_{L_X} G$, since by Proposition \ref{birationalRZ} the Zariski--Riemann space of  $X\sslash_{L_X} G$ itself does; as explained in the introduction, this is the key advantage of our approach over traditional VGIT.

\section{The infinitesimally stable locus}\label{sect_compact}

Motivated by the interpretation of GIT quotients $X\sslash_L G$ as providing a compactification of the quotient of the stable locus $X^s$, we next provide an analogous description for $V^G \underline X$. We work in the analytic topology and hence take the analytification $X\an$ of $X$; since analytifications of GIT quotients over $X$ form a projective system, we can also take the projective limit 
$$V^G\underline{X}\an:=\varprojlim (Y\sslash_{L_Y}G)\an,$$ which is a compact Hausdorff space by the general theory of projective limits of topological spaces.

Let us begin with the following definition:
\begin{definition}We say that a point $\underline x\in V^G \underline X\an$ is \textit{infinitesimally stable} over $X$ if for all GIT quotients $Y\sslash_{L_Y}G$ over $X$, there is a GIT quotient $Z\sslash_{L_Z} G$ over $X$ dominating $Y\sslash_{L_Y}G$ such that the realisation of $\underline x$ in $(Z\sslash_{L_Z} G)\an$ is stable.  \end{definition}

We recall that the realisation of a point $\underline x \in V^G \underline X\an$ is the image of $\underline x$ under the projection map
$$\pi_{L_Y}:V^G \underline X\an\to (Y\sslash_{L_Y} G)\an.$$ Thus the condition that $\underline x$ be infinitesimally stable asks that the collection of GIT quotients for which the realisation of $\underline x$ is stable is cofinal in the projective system.

An interpretation of infinitesimal stability is as follows: while an infinitesimally stable point $\underline x$ may not come from a point that is stable on $X$, it becomes stable after enough blow-ups. In that sense, $\underline x$  becomes stable once we unravel enough infinitesimal data associated to it (where infinitesimal is meant in the same sense that blow-ups include ``infinitesimal directions'').

\begin{definition}
The set of all infinitesimally stable points over $X$ will be called the \textit{infinitesimally stable locus}
$$X^{is}\sslash G\subset V^G\underline X\an.$$
(Note that we have not actually defined $X^{is}$ itself, despite our notation.)
\end{definition}

\begin{proposition}\label{coro_bij}There is a natural bijection
$$X^{is}\sslash G\simeq \bigcup_{Y\sslash_{L_Y} G}\pi_{L_Y}\mi((Y^s\sslash_{L_Y}G)\an).$$
That is, the infinitesimally stable locus is exactly the union of all preimages of stable loci of GIT quotients over $X$ by the projection maps of the projective limit $V^G \underline X$.
\end{proposition}
\begin{proof}
We first claim that any point $\underline x$ with $\pi_{L_Y}(\underline x)\in Y^s\sslash_{L_Y}G$ belongs to $X^{is}\sslash G$. Indeed, by Theorem \ref{thm_compstable} $(ii)$ it follows that for any $G$-equivariant blow-up $(\Bl_V Y)\sslash_{L-tE_V} G$ there is an inclusion $\pi^{-1}(Y^s) \subset (Bl_V Y)^s$ for all $t$ sufficiently small (where the $t$-dependence of the polarisation and hence the stable locus is implicit in the notation and $E_V$ is the exceptional divisor). It follows that $\pi_{L_Y}(\underline x)\in (Bl_V Y)^s\sslash_{L_Y}G$, and since such blowups are cofinal in the projective system of GIT quotients by the proof of Theorem \ref{thm_uvgit}, it follows that $\underline x$ is infinitesimally stable.

Conversely, if $\underline x$ is infinitesimally stable, then there exists a GIT quotient $Y\sslash_{L_Y} G$ with $\pi_Y(\underline x)\in Y^s\sslash_{L_Y}G$, proving the reverse inclusion.
\end{proof}

We are now equipped to prove Theorem B:
\begin{theorem}\label{thm_compact}

The set of infinitesimally stable points $X^{is}\sslash G$ is a dense, open subset of the compact Hausdorff topological space $ V^G\underline X$.  Thus the space $V^G\underline{X}\an$ is a compactification of the open subset $X^{is}\sslash G$ of infinitesimally stable points over $X$.  \end{theorem}

\begin{proof}

For a fixed GIT quotient $Y\sslash_{L_Y}$ over $X$, the stable locus $Y^s\sslash_{L_Y}G$ is a Zariski-dense open subset of $Y\sslash_{L_Y}G$. By GAGA, the image $(Y^s\sslash_{L_Y}G)\an$ of $Y^s\sslash_{L_Y}G$, under the analytification map sending $\bbc$-points of $X$ to the Serre analytification $X\an$, is similarly dense and open in $X\an$. 

We first prove openness of $X^{is}\sslash G$ in $V^G \underline X\an$. By Proposition \ref{coro_bij}, the set of infinitesimally stable points over $X$ is exactly given by the union of the sets
$$(\pi_{L_Y}\an)\mi (Y^s\sslash_{L_Y}G)\an$$
for all GIT quotients over $X$, where $\pi_{L_Y}\an$ is the analytification of the projection map $\pi_{L_Y}:V^G\underline X\to Y\sslash_{L_Y}G$. These projection maps are continuous, so each set $(\pi_{L_Y}\an)\mi (Y^s\sslash_{L_Y}G)\an$ is open, which implies that their union $$X^{is}\sslash G\simeq \bigcup_{Y\sslash_{L_Y} G}\pi_{L_Y}\mi((Y^s\sslash_{L_Y}G)\an)$$ is also an open subset of $V^G\underline X$.

We next prove density.  Since analytifications of GIT quotients are compact (again by GAGA), Proposition \ref{prop_projlimdensity} shows each such preimage defines a dense open set in $V^G\underline X\an$, hence so does their union. In other words, $(X^{is}\sslash G)\an$ embeds densely into $V^G\underline X\an$. 

Finally, it is a general fact in the theory of projective limits of topological spaces that $V^G\underline{X}\an$ is compact \cite[Chapitre 1, \textsection 9, Proposition 8]{bourbakitop14}, so the inclusion $$(X^{is}\sslash G)\an \subset  V^G\underline{X}\an$$ is an inclusion of a dense open subset into a compact space, meaning  $V^G\underline X\an$ can truly be viewed as a compactification and concluding the proof. 
\end{proof}

Yet another way to state this result is, \textit{via} Theorem \ref{thm_zr}, that the Zariski--Riemann space of any GIT quotient over (or of) $X$ is a compactification of the set of infinitesimally stable points over $X$.
\bibliographystyle{alpha}
\bibliography{bib}

\end{document}